\documentclass[11pt]{article}
\usepackage{amssymb}
\usepackage{multicol}
\usepackage{graphicx}
\usepackage{amsthm}
\usepackage{verbatim}
\usepackage{float}
\usepackage{forest,adjustbox}

\usepackage{pgfplots}
\pgfplotsset{compat=1.14}

\usepackage{pgfplots}
\usepackage{amsmath, epsfig}
\newtheorem{theorem}{Theorem}
\newtheorem{lemma}{Lemma}
\newtheorem{proposition}{Proposition}

\newcommand{\p}{\mathbb{P}}

\renewcommand{\phi}{\varphi}

\newcommand{\ignore}[1]{}
\theoremstyle{plain}
\newtheorem{thm}{Theorem}[section]

\newtheorem*{ques*}{Question}

\theoremstyle{definition}
\newtheorem{exmp}[thm]{Example}

\theoremstyle{remark}

\usepackage{color}
\definecolor{red}{rgb}{1,0,0}

\definecolor{blue}{rgb}{0,0,.7}

\title{An Extremal Problem on Rainbow Spanning Trees in Graphs}
\author{
Matthew DeVilbiss\thanks{Department of Mathematics, Statistics, and Computer Science. University of Illinois at Chicago.  {\tt mdevil2@uic.edu}.  This author's research supported in part by  NSF grant DMS-1343651.} \and Bradley Fain\thanks{Department of Mathematics. University of Delaware.  {\tt bfain@udel.edu}.  This author's research supported in part by  NSF grant DMS-1343651.} \and Amber Holmes\thanks{Department of Mathematics. University of Kentucky. {\tt Amber.Holmes@uky.edu}. This author's research partially supported by  NSF grant DMS-1343651.}\and Paul Horn\thanks{Department of Mathematics. University of Denver. {\tt paul.horn@du.edu}.  This author's research partially by NSF grant DMS-1343651, Simons Collaboration Grant 525039, and a University of Denver Internationalization Grant.} \and Sonwabile Mafunda\thanks{Department of Mathematics and Applied Mathematics. University of Johannesburg.  {\tt smafunda@uj.ac.za}.  This author's research partially supported by  NSF grant DMS-1343651 and the National Research Foundation in South Africa.} \and K.E. Perry\thanks{Mathematics. Soka University of America. {\tt kperry@soka.edu}.  This author's research partially supported by NSF grant DMS-1343651 and a University of Denver Internationalization Grant.}}

\begin{document}
 \maketitle 
 
  \begin{abstract}	
A spanning tree of an edge-colored graph is rainbow provided that each of its edges receives a distinct color. In this paper we consider the natural extremal problem of maximizing and minimizing the number of rainbow spanning trees in a graph $G$.  Such a question clearly needs restrictions on the colorings to be meaningful.  For  edge-colorings using $n-1$ colors and without rainbow cycles, known in the literature as JL-colorings, there turns out to be a particularly nice way of counting the rainbow spanning trees and we solve this problem completely for JL-colored complete graphs $K_n$ and complete bipartite graphs $K_{n,m}$.  In both cases, we find tight upper and lower bounds; the lower bound for $K_n$, in particular, proves to have an unexpectedly chaotic and interesting behavior. We further investigate this question for JL-colorings of general graphs and prove several results including characterizing graphs which have JL-colorings achieving the lowest possible number of rainbow spanning trees.  We establish other results for general $n-1$ colorings, including providing an analogue of Kirchoff's matrix tree theorem which yields a way of counting rainbow spanning trees in a general graph $G$.

\end{abstract}

\section{Introduction} 

Let $G$ be a (not necessarily properly) edge-colored simple graph with $|V(G)| = n$.  A rainbow spanning tree (RST) in $G$ is an acyclic, connected, spanning subgraph such that the color of every edge is distinct.  Given a coloring $\varphi: E(G) \to \mathbb{N}$ let 
\[
\mathcal{R}(G,\varphi) = \{T \subseteq E(G): \mbox{$T$ is a rainbow spanning tree}\}.  
\] 
The study of rainbow spanning trees in complete graphs, and more general graphs, has attracted a great deal of attention lately, especially on work related to the Brualdi-Hollingsworth conjecture which posits that if the edges of $K_{2n}$ are colored via a one-factorization then the edge set can be partitioned into edge-disjoint RSTs.  See \cite{BH, CHH, FuLoPerryRodger, glock, Horn, HN, MPS, PS} for the conjecture and some recent developments along these lines.

In this paper we are concerned with a natural extremal problem regarding rainbow spanning trees: maximizing and minimizing $|\mathcal{R}(G,\varphi)|$ over a collection of colorings. One immediately notes that the problem, without restrictions on the colorings, is not interesting: any coloring with fewer than $n-1$ colors cannot possibly contain a rainbow spanning tree so for such a coloring, $|\mathcal{R}(G,\varphi)| = 0$.  On the other hand, if all edge colors are distinct, the number of RSTs is simply the number of trees in the graph.  This can be easily computed by the matrix tree theorem of Kirchoff (see \cite{k}) for a general graph $G$ and is $n^{n-2}$ by Cayley's formula for the special case where $G= K_n$ (see \cite{c}).

To make the problem interesting and non-trivial, and in the spirit of anti-Ramsey results, we consider this extremal problem on a certain class of colorings, known in the literature as {\it JL-colorings} \cite{ghnp, hhjo, jz,jo}. A coloring $\varphi: E(G) \to [n-1]$ is a JL-coloring if it is surjective and rainbow cycle free.  Note that these properties are rather delicately balanced with respect to an interplay between RSTs and cycles: if $n$ colors appear in an edge coloring, then $G$ necessarily contains a rainbow cycle, but if fewer than $n-1$ colors appear in an edge coloring, then no RSTs can exist.  

Given a JL-coloring $\varphi$, let $\mathcal{C}_1, \dots, \mathcal{C}_{n-1}$ denote the color classes of $\varphi$.  If a single edge of each color is selected, this gives $n-1$ edges of distinct colors; further, since $\varphi$ is a rainbow cycle free coloring, this collection of edges yields a rainbow spanning tree.  This simple observation means that for a JL-coloring $\varphi$,
\begin{equation}
|\mathcal{R}(G,\varphi)| = \prod_{i=1}^{n-1} |\mathcal{C}_i|.
\label{product}
\end{equation}
Further, since $\sum |C_i| = |E(G)|$, convexity immediately implies that
\begin{equation}
(|E(G)| - (n-2)) \cdot 1^{n-2} \leq  |\mathcal{R}(G, \varphi)| \leq \left( \frac{|E(G)|}{n-1} \right)^{n-1}.  \label{convexity!}
\end{equation}
How good are these particular estimates?  While both can be tight (simultaneously, in the case where $G$ is itself a tree), for the interesting special case where $G = K_{n}$ they are both far from tight.  In particular, we prove that
\begin{theorem} \label{tmain} 
Let $\varphi:E(K_n) \to [n-1]$ be a JL-coloring.  Then,
\[
2^{2n-O(\log n)} = \frac{\mu(n)}{n}  \leq |\mathcal{R}(K_n,\varphi)| \leq (n-1)!
\]
where $\mu(n)$ has the defining property that if $s$ is the unique power of 2 such that $\frac{n}{3} \leq s < \frac{2n}{3}$ then,
\[
\mu(n) = n \cdot \mu(s) \cdot \mu(n-s) 
\]
and $\mu(1)= 1$.  Both inequalities have colorings $\varphi$ for which they are tight.
\end{theorem}

As we shall see, this gives a surprisingly (to us) chaotic lower bound for $|\mathcal{R}(G,\varphi)|$ (cf. Figure 3 in Section \ref{bfig} and the surrounding discussion), which grows exponentially in $n$, as opposed to the trivial linear lower bound in the inequality in \eqref{convexity!}. For $n \leq 14$, this evaluates to the lower bounds given below:

\[\begin{array}{|c||c|c|c|c|c|c|c|c|c|c|c|c|c|c|}
\hline 
n= & 2 & 3 & 4& 5&6&7&8&9&10&11&12 & 13 & 14 \\\hline \hline
|\mathcal{R}(K_n,\varphi)| \geq & 1 & 2 & 4 &12 & 32 &96& 256& 960& 3072 & 10752 &32768 & 122880 & 393216 \\
\hline
\end{array} 
\]

We further study the extremal problem on complete bipartite graphs, proving
 
 \begin{theorem} \label{tbip}
 	Let $\varphi:E(K_{n,m}) \to [nm-1]$ be a JL-coloring. Then for $n \leq m$,
 	\[
 	(n-1)(m-1) + 1 \leq |\mathcal{R}(K_{n,m},\varphi)| \leq m^{n-m+1}((m-1)!)^2.
 	\]
 	
 	Both inequalities have colorings $\varphi$ for which they are tight.
 \end{theorem}

Particularly interesting here, to us, is the stark difference between this case and the case of $K_n$ in terms of the proof mechanics: in particular, the lower bound -- difficult in $K_n$ -- is now the trivial bound, while the upper bound -- quite easy in the $K_n$ case -- is comparatively more difficult.

Finally, we consider some related problems: What happens if we work with more general graphs and/or more general colorings?  Here, we are able to characterize graphs with JL-colorings for which the trivial lower bound from $(ii)$ is tight and we prove an analogue of the matrix tree theorem counting rainbow spanning trees in general graphs that may be of interest in future investigations along these lines for non-JL colorings (cf. Theorem \ref{MTT}). 

The remainder of the paper is organized as follows: In the next section we introduce a particularly nice way of thinking about JL-colorings which allows us to derive our bounds.  We then turn our attention to the proof of Theorem \ref{tmain} in Section \ref{RST_Kn} before proving Theorem \ref{tbip} in Section \ref{RST_Knm}. We conclude with results concerning general graphs and colorings, some open questions, and directions for future work.

\section{The Structure of JL-Colorings} \label{structure}

Recall that a JL-coloring $\varphi$ is a rainbow cycle free $(n-1)$-edge-coloring of a graph $G$ with order $n$. There is a representation of JL-colorings as labeled binary trees which we will use to count RST's in a given JL-colored graph.  The key to this approach, which has appeared in a series of papers of Johnson and collaborators (see \cite{ghnp,hhjo, jz,jo}), is the following proposition:
\begin{proposition} 
	Suppose $\varphi$ is a JL-coloring of a connected graph $G$.  Then there is a  partition of $V(G)$ into sets $A$ and $\bar{A}$, so that $e(A,\bar{A})$, the set of edges between vertices in $A$ and $\bar{A}$, is monochromatic in $\varphi$, and both $\varphi|_{A}$ and $\varphi|_{\bar{A}}$ are JL-colorings of the graphs induced on $A$ and $\bar{A}$ respectively.   \label{prop1} 
\end{proposition} 

\noindent We note here that JL-colorings of graphs only exist for connected graphs so this proposition implies that both $A$ and $\bar{A}$ are connected and that the cut contains at least one edge, but further that the cut is an entire color class and that $\varphi|_{A}$ and $\varphi|_{\bar{A}}$ have disjoint color sets.


This was originally proved for complete graphs in {\cite{ghnp}}, complete bipartite graphs in {\cite{jz}}, and finally for complete multipartite graphs in {\cite{jo}}.  Recently, it has been established for arbitrary JL-colored graphs in {\cite{hhjo}}. Iterating Proposition \ref{prop1} on the induced subgraphs gives iteratively nested subsets so that each non-trivial subset $A$ is partitioned into two subsets $A'$ and $\bar{A}'$ where the edges between the subsets are monochromatic and the coloring induced on each is a JL-coloring.  

This allows us to create a rooted binary tree with $n-1$ internal vertices from every JL-coloring.  Here, each vertex is labeled with sets: the root is labeled with $V(G)$ and the children of a vertex labeled $A$ are the two JL-colored subsets $A'$ and $\bar{A}'$ guaranteed by Proposition \ref{prop1}.

It is easy to see that this construction of a tree from a JL-coloring actually gives a correspondence between JL-colorings of a graph and subgraph-labeled binary trees with $n-1$ internal vertices, where each subgraph is connected and the label of any internal vertex is partitioned by the labels of its two children. The colors of the corresponding JL-coloring can be associated with the internal vertices so that the edges of a color are exactly the edges between the two children of the associated internal vertex.

These representations can be simplified for the two main graph classes considered in this paper: complete graphs and complete bipartite graphs, and we do so below.  	

\subsection{JL-Colorings of $K_n$} 
\label{jlkn}

For the case where $G=K_{n}$, the exact sets labeling the vertices in the associated tree make no difference when enumerating rainbow spanning trees: only the number of vertices in each label matters. Thus, a JL-coloring of $K_n$ is equivalent (up to vertex labeling) to a rooted binary tree with $n-1$ internal vertices, so that the root is labeled by $n$ and the two children of a vertex labeled $r \geq 2$ are labeled $p$ and $q$ with $r=p+q$, $p, q \geq1$, and all $n$ leaves are labeled $1$. We call such a tree a {\it JL-tree}. Equivalently,  a JL-tree is a rooted binary tree in which every vertex is labeled with the number of leaves below (or including) itself. Further, there is a bijectin between JL-trees and JL-colored $K_n$'s.

As a clarifying example, we illustrate a JL-coloring and its respective JL-tree for $K_5$.

\begin{figure}[h!] 
	\begin{center}
		\centering
		\def\svgwidth{12cm}
		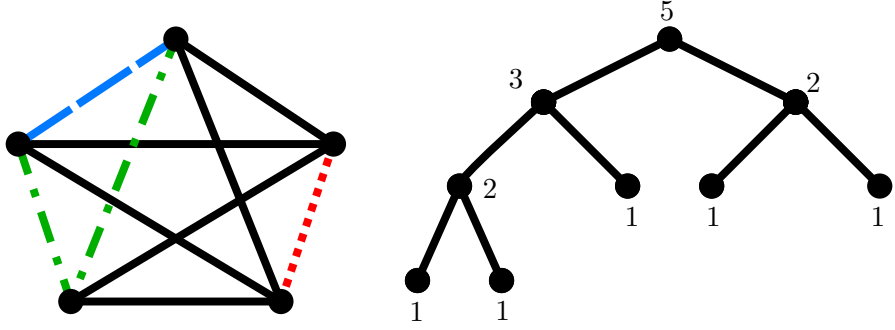
		
		\caption{A JL-coloring of $K_5$ and its associated JL-tree.  } 
	\end{center}
\end{figure}

\subsection{JL-Colorings of $K_{n,m}$}
\label{jlknm}

If $G=K_{n,m}$, then the trees described above can also be simplified.  In this case, the connected subgraphs $A$ and $\bar{A}$ are smaller complete bipartite graphs. The tree is determined by the size of each bipartite label and from which part in the parent label each smaller part originates.

\begin{figure}[h!] 
	
	\centering
	\begin{center}
		\def\svgwidth{11cm}
		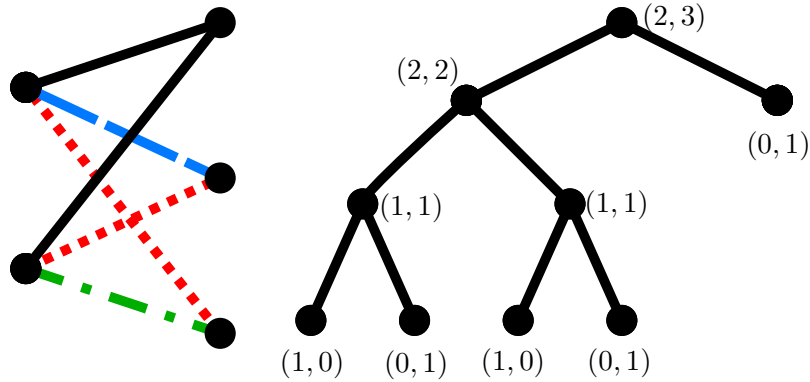
		\caption{A JL-coloring of $K_{2,3}$ and its associated $JL_b$-tree.} 
	\end{center} \label{f2} 
\end{figure}

  In light of this, a JL-coloring of $K_{n,m}$ is equivalent to a rooted full binary tree with $n+m-1$ internal vertices so that the root is labeled $(n,m)$ and the children of a vertex labeled $(r_1,r_2)$ are labeled $(p_1,p_2)$ and $(q_1,q_2)$ so that $p_1 + q_1 = r_1$, $p_2+q_2=r_2$ with the $p_i$, $q_i$ non-negative and so that if $p_1=0$ then $p_2 = 1$ (respectively if $p_2=0$, then $q_1=1$).  This last restriction is because a single vertex -- $K_{1,0}$ -- is connected, but $K_{2,0}$ is not.  Note that the vertices labeled $(1,0)$ or $(0,1)$ are exactly the leaves of the tree.  We call such a tree a {\it $JL_b$-tree} and again observe that there is a bijection between $JL_b$-trees and JL-colored $K_{m,n}$'s.

  An example of a $JL_b$-tree for $K_{2,3}$ is given in Figure 2.

\section{Rainbow spanning trees in $K_n$} \label{RST_Kn}

We begin by considering the case where the graph $G$ is complete.  In this instance, we observe that the JL-tree (introduced above in Section \ref{jlkn}) captures not only the structure of the JL-coloring, but also the number of rainbow spanning trees in the coloring.  

Since the graph is complete, the number of edges with a given color associated with an inner vertex $r$ is the product of the sizes of its two children, $p$ and $q$.  It follows that multiplying the sizes of all color classes together in a JL-coloring of $K_n$ (as in $(i))$ is equivalent to taking the product of all non-root labels of its associated JL-tree (or, equivalently, finding the product of all labels of the associated tree and dividing by $n$).  

\subsection{The Upper Bound}

We first turn our attention to the upper bound in Theorem \ref{tmain}.  This turns out to be relatively simple after the discussion above.  We prove that the JL-tree maximizing the product is the one where the two children of a vertex labeled $r$ are labeled $r-1$ and $1$.  

\begin{proof}[Proof of the upper bound in Theorem \ref{tmain}]
	We prove that the JL-tree maximizing the product is as described in the paragraph above: a tree where the vertex labeled $r$ has children labeled $r-1$ and $1$.  Such a tree has product $n!$ and hence, describes a coloring with $(n-1)!$ RSTs.  We proceed by induction on $n$, noting that it is trivially true for $n=1$. Now, suppose that in a maximizing tree, a vertex $r$ is split as $p$ and $q$ with $p, q \geq 1$: By the inductive hypothesis, the labels below the vertex labeled $r$ have product at most $p!q!$, but it is easy to see that $p!q! \leq (r-1)!$ if $p+q = r$ as this is equivalent to the statement that $\binom{r}{p} \geq r$ for $1 \leq p \leq r-1$.  Thus, the optimal split is $p=1$ and $q=r-1$, and the result follows. 
\end{proof} 
\subsection{The Lower Bound} 
We now turn to the significantly harder case of the lower bound.  Since the upper bound was obtained by taking the splits in the JL-tree to be as unbalanced as possible one might expect, or hope, that the lower bound would be achieved by taking the split to be as balanced as possible, namely a vertex labeled $n$ should split as $\lfloor \tfrac{n}{2}\rfloor$ and $\lceil \tfrac{n}{2} \rceil$. While this holds for powers of 2, it turns out to be false in general: one part of the optimal split is always a power of two, specifically the unique power of two between $\tfrac{n}{3}$ and $\tfrac{2n}{3}$.  To show this we study the following function.  

  For $n \in \mathbb{N}$, let 
\begin{equation}
\mu(n) = \min_{1 \leq p \leq n-1} n \cdot \mu(p) \cdot \mu(n-p), \label{eqn:mu}
\end{equation}
and let $\mu(1) =1$.  

This function corresponds to $n$ times the minimum number of RSTs.  This can be seen by noticing that if one takes an interior vertex of a JL-tree, as well as the vertices below it, one obtains a JL-tree for a smaller complete graph.  Thus, $\mu(n)$ is taking the product of all of the labels of the vertices of our `minimum' JL-tree recursively.

In light of this, we are interested in proving the following theorem, which is the lower bound of Theorem \ref{tmain}:
\begin{theorem} 
	Let $s$ denote the unique power of 2 so that $\frac{n}{3} \leq s < \frac{2n}{3}$.  Then, 
	\[
	\mu(n) = n\mu(s) \mu(n - s).  
	\]    \label{t1}
\end{theorem} 
\noindent{\bf Remark:} This does not quite finish the stated bound in Theorem \ref{tmain} that
\[
\frac{\mu(n)}{n} = 2^{2n-O(\log n)};
\]
the final step in this equality is recorded in Proposition \ref{pfin} at the conclusion of this section.

In order to prove Theorem \ref{t1}, we first introduce the following continuous analogue of $\mu$.  For $x \geq 1$, consider the function
\begin{equation}
\tau(x) = \frac{2^{2x - 2}}{x}. \label{eqn:tau}
\end{equation} 

It is not necessarily obvious that $\tau$ is, in any sense, a continuous analogue of $\mu$.  To this end, note that
\[
\log_2 \tau(x) = 2x - \log_2(x) - 2 
\]
is a convex function of $x$.  This log convexity means that for $x \geq 2$,
\begin{align*}
\min_{1 \leq p \leq x-1} x \cdot\tau(p)\cdot\tau(x-p) &= x \tau(x/2)^2 \\&= x \frac{2^{2x - 4}}{(x/2)^2} \\
&= \frac{2^{2x-2}}{x} = \tau(x),
\end{align*}
so that $\tau(x)$ satisfies the defining property $\eqref{eqn:mu}$ of $\mu$ while extending the minimization to all real numbers as opposed to merely integers, and $\tau(1)=\mu(1)=1$.  We now make some elementary observations.

\noindent {\bf Claim 1:} For all integers $n \geq 1$, $\mu(n) \geq \tau(n)$.
\begin{proof}
	
	To see this, proceed by induction.  Equality holds if $n=1$, and for $n \geq 2$ note that for some $1 \leq p \leq n-1$,
	\[
	\mu(n) = n \cdot \mu(p)\cdot \mu(n-p) \geq n \cdot \tau(p)\cdot \tau(n-p) \geq n \cdot  \tau(n/2) \cdot \tau(n/2) = \tau(n).  
	\]
\end{proof}

\noindent{\bf Claim 2:} For all integers $i \geq 0$, $\mu(2^{i}) = \tau(2^{i})$.
\begin{proof} 
	This is also shown by induction. Equality holds for $i = 0$ and for $i \geq 1$ observe that,
	\[\mu(2^{i}) \geq \tau(2^{i}) = 2^{i}\tau(2^{i-1})^2 = 2^i\mu(2^{i-1})^2 \geq \mu(2^i).
	\]

	Here, the first inequality is Claim $1$ and the final inequality is from the definition of $\mu$ \eqref{eqn:mu}. Combined, the inequalities force equality and complete the inductive step.  
\end{proof}

We remark here that Claims 1 and 2, in fact, prove the lower bound in Theorem \ref{tmain} is achieved since $\frac{\mu(n)}{n} \geq \frac{\tau(n)}{n} = 2^{2n-2\log_2 n-2}$ and $\mu(n) = \tau(n)$ when $n$ is a power of 2. It  remains to show that $\mu(n)$ is always of the order $2^{2n-O(\log n)}$.  This is done in Proposition 2 at the end of this section.


Ultimately, we are interested in the relationship between $\mu(n)$ and $\tau(n)$; to this end let
\begin{equation}
\beta(n) = \frac{\mu(n)}{\tau(n)}. \label{dbeta}
\end{equation}

To get a sense of values of $\beta, \mu$, and $\tau$, we include some values of them below:

\[\begin{array}{|c||c|c|c|c|c|c|c|c|c|c|c|c|c|c|}
\hline 
n= & 2 & 3 & 4& 5&6&7&8&9&10&11&12 & 13 \\\hline \hline
\mu(n) & 2 & 6 & 16 & 60 & 192 &672& 2048& 8640& 30720 & 118272 &393216 & 1597440  \\
\hline
\tau(n) & 2 & 5\frac{1}{3} & 16 & 51\frac{1}{5}&170\frac{2}{3}&585\frac{1}{7}&2048&7281\frac{7}{9} & 26214\frac{2}{5} & 95325\frac{1}{11} & 349525\frac{1}{3} & 129055\frac{1}{13}\\
\hline
\beta(n) & 1 & \frac{9}{8} & 1 & \frac{75}{64} & \frac{9}{8} & \frac{147}{128} &1& \frac{1215}{1024} & \frac{75}{64} & \frac{2541}{2048} & \frac{9}{8} & \frac{2535}{2048}\\
\hline
\end{array} 
\]
 
By Claim 1, we know that $\beta(n) \geq 1$ for all $n$.  A straightforward calculation reveals that if $\mu(n) = n \mu(p)\mu(n-p)$, then 

\begin{equation}
\beta(n) = \frac{\mu(n)}{\tau(n)} = \frac{n^2 \cdot \mu(p)\mu(n-p)}{2^{2n-2}} = \frac{n^2}{4p(n-p)}\beta(p)\beta(n-p), \label{eq:beta}
\end{equation}
and finding the minimizing split that defines $\mu(n)$ is equivalent to finding the value of $p$ that minimizes \eqref{eq:beta}. To that end, we now proceed with the proof of Theorem \ref{t1} which will show that if $s$ is the unique power of 2 so that $\frac{n}{3} \leq s < \frac{2n}{3}$, then $s$ is the value of $p$ that minimizes \eqref{eq:beta}.  

We remark here that working with $\beta$ proves to be a bit simpler than dealing with $\mu$ directly, as we at least have some information (and some clue as to why the powers of two occur): $\beta(n) \geq 1$ with $\beta(2^{n}) = 1$, so minimizing the product \eqref{eq:beta} `prefers' powers of two.  Unfortunately, this is not enough to complete the proof as the $\beta$ function is quite chaotic and $\limsup \beta(n) = \infty$, and there is no immediate reason that the product of two numbers larger than one may not be smaller than the product of one and a rather larger number. (Note that this is not a priori clear from the definition. It follows, however, rather easily from the expression \eqref{eq:beta} and Theorem \ref{t1} by taking an appropriate subsequence.) As an illustration, we present in Figure 3 the values of $\beta(n)$ for $1 \leq n \leq 256$.  

\begin{figure}[h!]
\label{bfig}
\begin{tikzpicture}
\begin{axis}[
width=\textwidth,
height=2.5in,
xmin=0,
xmax=257,
ymin=1
]
\addplot[black,mark=*,mark size={.3mm}] table {data3.txt};
,\end{axis}
\end{tikzpicture}
\caption{A plot of $\beta(n)$ for $1 \leq n \leq 256$ (linearly interpolating between points).} 
\end{figure}
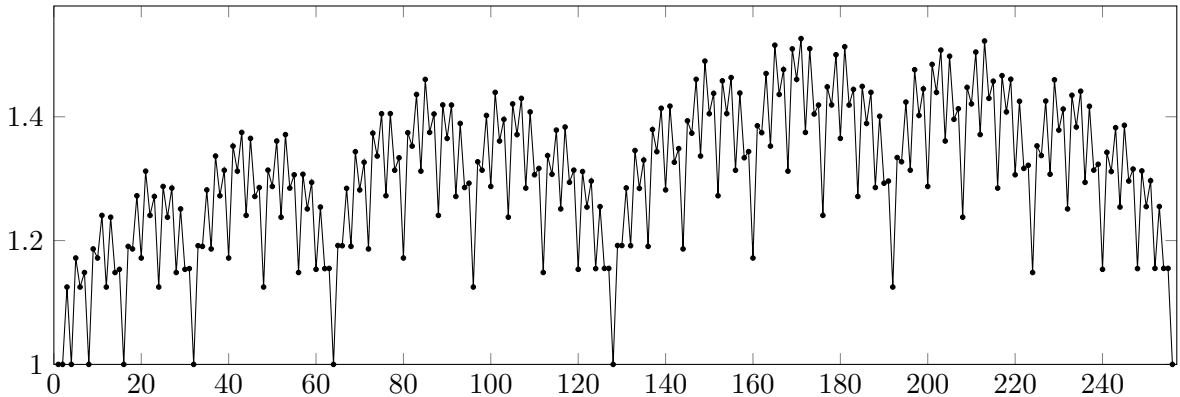

Several striking features of $\beta(n)$ appear in this picture: for instance it appears that $\beta(n)$ has some self-similarity properties, and alternates between increasing and decreasing.  Both of these turn out to be true: it is not difficult to verify that $\beta(2n) = \beta(n)$, and a more involved argument shows that $\beta(n)$ for even $n$ is smaller than $\beta(n-1)$ and $\beta(n+1)$.  These facts, however, turn out to be not important for solving the recurrence, so we do not record their (rather laborious, in the second case) proofs here.

\begin{proof}[Proof of Theorem \ref{t1}]
	We proceed by induction.  Theorem $\ref{t1}$ holds for $n=1$ so let us assume that it holds for all integers $k < n$. We shall prove that it holds for $n$.   
	
	For any positive integers $p, q$ with $p+q = n$, let 
	\[
	\beta(p,q) = \frac{n^2}{4\cdot p \cdot q}  \beta(p)\beta(q).
	\]
	We want to show that $\beta(p,q) \geq \beta(s,n-s)$ where $s$ is unique power of 2 with $\frac{n}{3} \leq s < \frac{2n}{3}$.  
	
	We first prove the following lemma:
	
	\begin{lemma} \label{lcomparing}
		Let $p, q$ be positive integers with $p+q = n$ and $p \leq q$.  Let $(p_1, p_2)$ and $(q_1, q_2)$ be optimal minimizing splits of $p$ and $q$ respectively, with the two numbers in the splits ordered arbitrarily.  Then:	
		\begin{enumerate}
			\item[(a)] \[
			\beta(p,q) \geq \frac{pq}{(p_1+q_1)(p_2+q_2)} \beta(p_1+q_1,p_2+q_2) 
			\]
			\item[(b)] \[
			\beta(p,q) \geq \frac{q}{p+q_1} \beta(p+q_1,q_2) 
			\]
		\end{enumerate}
	\end{lemma}
	\begin{proof} 
		To prove $(a)$, we observe:
		\begin{align*}
		\beta(p,q) &= \frac{n^2}{4 \cdot p \cdot q} \beta(p)\beta(q) \\
		&= \frac{n^2 \cdot p \cdot q}{4^3 \cdot p_1 \cdot p_2 \cdot q_1 \cdot q_2} \beta(p_1)\beta(p_2)\beta(q_1)\beta(q_2) \\
		&=  \frac{n^2 \cdot p \cdot q (p_1+q_1)^2(p_2+q_2)^2}{4^3 \cdot p_1 \cdot p_2 \cdot q_1 \cdot q_2\cdot (p_1+q_1)^2(p_2+q_2)^2} \beta(p_1)\beta(p_2)\beta(q_1)\beta(q_2) \\
		&= \frac{n^2 \cdot p \cdot q}{4 \cdot (p_1 + q_1)^2 (p_2 + q_2)^2} \cdot \frac{(p_1 + q_1)^2}{4 \cdot p_1 \cdot q_1} \beta(p_1) \beta (q_1) \cdot \frac{(p_2+q_2)^2}{4 \cdot p_2 \cdot q_2} \beta (p_2) \beta (q_2) \\
		&\geq \frac{n^2 \cdot p \cdot q}{4 \cdot (p_1+q_1)^2(p_2+q_2)^2} \beta(p_1+q_1)\beta(p_2+q_2)\\
		&= \frac{p \cdot q}{(p_1+q_1)(p_2+q_2)}  \cdot \frac{n^2}{4 \cdot (p_1+q_1)(p_2+q_2)} \beta(p_1+q_1)\beta(p_2+q_2)\\&= \frac{p \cdot q}{(p_1+q_1)(p_2+q_2)}  \beta(p_1+q_1, p_2+q_2)
		\end{align*}
		
		where the inequality comes from the fact that $\beta(p_1 + q_1)$ and $\beta(p_2 + q_2)$ might be suboptimal splits for $p_1 + q_1$ and $p_2 + q_2$, respectively.  The proof of $(b)$ follows in the same manner, only splitting $q$ instead of both $q$ and $p$.	
		
	\end{proof}
We now proceed by comparing an arbitrary split $p+q=n$, $p \leq q$, to our conjectured optimal split, $s + t = n$, where $s$ is the unique power of 2 so that $\frac{n}{3} \leq s < \frac{2n}{3}$. (For the remainder of this section, let $s$ and $t$ be defined as such.) We consider two cases: the first in which $(p,q)$ is a more balanced split than $(s,t)$ and Case 2 in which $(p,q)$ is less balanced. In both cases, we show $\beta(p,q) \geq \beta(s,t)$, thus proving the theorem.  To clarify the mechanics of this section, we give a brief example after the conclusion of the proof for the benefit of the reader.  
	
	\noindent{\bf Case 1 (More Balanced Split):} $\min(s,t) <  p \leq q < \max(s,t)$.
	
    Suppose that $\min(s,t) <  p \leq q < \max(s,t)$. Notice it follows that $pq > st$. Let $(p_1,p_2)$ and $(q_1,q_2)$ be the optimal splits of $p$ and $q$, respectively, and note that by induction, $p$ and $q$ split as conjectured; let $p_1$ and $q_1$ denote the powers of 2, respectively.  
    
    Note that since $p \leq q$ that $p_1 \leq q_1$.If $s \leq p \leq q$, then both $p_1, q_1 \geq s/2$.    Since $q < t \leq 2s$, $q_1 < s$ as $\frac{2}{3}q 2s$.  Note that $\frac{(p+q)}{3} = \frac{n}{3} \leq s$.  If $q \geq \frac{3}{2}s$, then this implies that $2p/3 \leq s$.  This, in turn, implies that $p_1 = s/2$ as we already know that $p_1 \geq s/2$, and in this case $p_1=s/2$ and $q_1=s$.  Otherwise, since $2q/3 < \frac{3}{2}s$ $p_1 = q_1 = s/2$.
    
    Similarly, if $s > t$, we have that $p_1 \leq q_1 \leq s/2$ and since $p > t > s/2$ we have that $s/4 \leq p_1 \leq q_1$.  Now, note that $\frac{2(p+q)}{3} = \frac{n}{3} > s/2$ which implies that if $p/3 \leq s/4$, then $q/3 > s/4$ so that the options here are $p_1=q_1=s/2$ or $p_1=s/4$ and $q_1 = s/2$.

	\noindent{\it Balancing:} 
	
	If $p_1 = q_1 = s/2$ then we apply Lemma \ref{lcomparing} to see that
	\[
	\beta(p,q) \geq \frac{pq}{(p_1+q_1)(p_2+q_2)} \beta(p_1 + q_1,p_2+q_2) = \frac{pq}{st} \beta(s,t) > \beta(s,t).  
	\]
	Here we use the fact that $p_1 + q_1 = s$ so $p_2+q_2 = t$, and the earlier observation that $pq > st$.
	
	\noindent{\it Balancing then Unbalancing :}
	
	Now suppose that either $p_1 = s/2$ while $q_1 = s$ or $p_1=s/4$ while $q_1 = s/2$. We proceed by first comparing $(p,q)$ to an intermediary more balanced split $(p',q')$ such that $p' = p_1+q_1$,  $q' = p_2+q_2$ and $p < p' \leq q' < q$.  In each of these cases, after applying Lemma \ref{lcomparing}, we obtain that 
	
	\[
	\beta(p,q) \geq \frac{pq}{(p_1+q_1)(p_2+q_2)} \beta(p_1 + q_1,p_2+q_2) = \frac{pq}{p'q'} \beta(p',q').  
	\]
	From here we get that $p'$ splits optimally as $s/2 + p_2'$ for some $p_2'$, and $q'$ splits optimally as $s/2 + q_2'$ for some $q_2'$.  But now another application of Lemma \ref{lcomparing} shows that
	\[
	\beta(p,q) \geq \frac{pq}{p'q'} \beta(p',q') 
	\geq \frac{pq}{p'q'} \frac{p'q'}{(s/2 + s/2)(p_2' + q'_2)} \beta(s/2 + s/2,p_2'+q'_2) = \frac{pq}{st} \beta(s,t) > \beta(s,t).    
	\]

	\noindent{\bf Case 2 (Less Balanced Split):} $p < \min(s,t) \leq \max(s,t) < q$.
	
	This case works much like Case 1, with somewhat of an opposite feel since the split we are considering is less balanced than our conjectured optimal split. To that end, suppose that $p < \min(s,t) \leq \max(s,t) < q$. As above, let $(p_1,p_2)$ and $(q_1, q_2)$ be the optimal splits of $p$ and $q$, respectively, and note that by induction, $p$ and $q$ split as conjectured; let $p_1$ and $q_1$ denote the powers of 2, respectively. It follows that $q_1 = s$ or $s/2$ since $q > s$.

	
	\noindent{\it Balancing:}
	
	If $q_1=s$, this is quite easy.  We apply Lemma \ref{lcomparing} directly to see that
	\[
	\beta(p,q) \geq \frac{q}{p+q_2} \beta(p+q_2,q_1) = \frac{q}{t} \beta(t,s) > \beta(s,t)  
	\]
	where we use the fact that $q > \max(s,t)$ and also the fact that $n = s + t = q_1 + p+q_2$ implies $t = p+q_2$.  
	
	\noindent{\it Unbalancing then Balancing:} 
	
	If $q_1 = s/2$ we first note that
	\[
	\beta(p,q) \geq \frac{q}{p+q_2} \beta(q_1, p+q_2). 
	\]
	Now, consider the optimal split $r_1, r_2$ of $p+q_2 = n-s/2$.  Observe that $s < 2n/3$, so $s/2 < n/3$, and hence $p+q_2 > 2n/3 \geq \max(s,t)$.  Thus, this split is also less balanced than $s,t$ (and potentially less balanced than $p,q$ so that the ratio $\frac{q}{p+q_2}$ appearing above may be less than one.)  None the less, we proceed noting that since $n > p+q_1 \geq s$, either $r_1 = s$ or $r_1 = s/2$.  
    
    \noindent If $r_1 = s$, then we again apply Lemma 1 to see that:   
	\[
	\beta(p,q) \geq \frac{q}{p+q_2} \beta(q_1, p+q_2) \geq \frac{q}{p+q_2} \cdot \frac{p+q_2}{q_1 + r_2} \beta(s,t) = \frac{q}{t} \beta(s,t) > \beta(s,t).
	\] 
	Otherwise, if $r_1 = s/2$, then we balance slightly differently:
	\[
	\beta(p,q) \geq \frac{q}{p+q_2} \beta(q_1, p+q_2) \geq \frac{q}{p+q_2} \cdot \frac{p+q_2}{q_1 + r_1} \beta(q_1 + r_1, r_2) = \frac{q}{s} \beta(s,t) > \beta(s,t).  
	\]
	In both cases, we see that $\beta(p,q) \geq \beta(s,t)$, thus proving the theorem.
	
\end{proof}

\begin{exmp}{\rm
To better understand the mechanics of the proof above, it is rather helpful to work through an example.  To that end, consider the case $n=187$, whose optimal split is $s=64$ and $n-s=123$.  To show this, we want to compare $(64,123)$ to an arbitrary split of 187.  Suppose we start with a more balanced split: $(90,97)$.  To compare these splits, we first compare the $(90,97)$ split to the more balanced split $(91,96)$, and then ultimately to the (less balanced, but optimal) $(64,123)$ split. Figures $3$, $4$, and $5$ illustrate the optimal split for $n=187$, along with the splits we compare them to which, by induction, we know split optimally below the first step.  

\begin{figure}[H]
\begin{minipage}{0.33\linewidth}
\centering
\begin{forest}
[187 [64 [32][32]]
[123 [64][59]]]
\end{forest}
\caption{The optimal split}
\end{minipage}
\begin{minipage}{0.33\linewidth}
\centering
\begin{forest}
[187 [90 [32][58]]
[97 [33][64]]]
\end{forest}
\caption{The $(90,97)$ split}
\end{minipage}
\begin{minipage}{0.33\linewidth}
\centering
\begin{forest}
[187 [91 [32][59]]
[96 [32][64]]]
\end{forest} 
\caption{The $(91,96)$ split}
\end{minipage}
\end{figure}

\begin{align*}
\beta(90,97) &=\frac{187^2}{4\cdot 90 \cdot 97} \beta(90)\beta(97)\\
 &= \frac{187^2\cdot 90\cdot 97}{4^3\cdot 32 \cdot 58\cdot 33\cdot 64} \beta(32)\beta(58)\beta(33)\beta(64)\\
 &\geq \frac{187^2\cdot 90\cdot 97}{4\cdot 91^2 \cdot 96^2} \beta(91)\beta(96)\\
 &= \frac{90 \cdot 97}{91\cdot 96}\left[\frac{187^2}{4\cdot 91\cdot 96} \beta(91)\beta(96)\right] \\
 &= \frac{90 \cdot 97}{91\cdot 96}\left[\frac{187^2\cdot 91\cdot 96}{4^3\cdot 32^2\cdot 59\cdot 64} \beta(32)^2\beta(59)\beta(64)\right] \\
 &= \frac{90 \cdot 97}{64\cdot 123}\left[\frac{187^2}{4\cdot 64\cdot 123} \beta(64)\beta(123)\right] \\
& > \frac{187^2}{4\cdot 64\cdot 123} \beta(64)\beta(123) = \beta(64,123)
\end{align*}
Hence, the $(90,97)$ split of $n=187$ is not an optimal split since there exists a split of $(64,123)$ with smaller $\beta$.
}
\end{exmp}

We conclude this section with a final proposition, completing the claimed bound that $\frac{\mu(n)}{n} = 2^{2n - O(\log n)}$ from Theorem \ref{tmain} for all $n$.

\begin{proposition} \label{pfin}
$\beta(n) \leq  n^{O(1)}$ and $\beta(n) \geq \frac{9}{8}$
for all $n$ which are not powers of two.  

\end{proposition} 
\begin{proof} 
Both statements follow from \eqref{eq:beta} and Theorem \ref{t1} which together show
\[
\beta(n) = \frac{n^2}{4 s (n-s)} \beta(s)\beta(n-s),
\]
where $s$ is the unique power for two satisfying $\frac{n}{3} \leq s < \frac{2n}{3}.$  Since $\beta(s) = 1$, one obtains that
\[
\beta(n) \leq \frac{9}{8} \beta(n-s).
\]
Since $n-s \leq \frac{2}{3}n$, this gives that $\beta(n) \leq (9/8)^{\log_{3/2}(n)}$ and, thus, proves the first statement. 

The second statement follows by strong induction, noting that it is true for $n=2^{i} + 2^{i+1}$ from the fact that for such integers, $\frac{n^2}{4\cdot s\cdot(n-s)} = \frac{9}{8}$, and for other integers, at least one of the terms appearing in the decomposition of $\beta(n)$ is not a power of two and hence, is at least $\frac{9}{8}$.  
\end{proof} 

\noindent{\bf Remark:} The fact that $\beta(n) \leq n^{O(1)}$ completes the claimed bound on $\frac{\mu(n)}{n}$ from Theorem \ref{tmain} as
\[
\frac{\mu(n)}{n} = \beta(n) \frac{\tau(n)}{n} = \beta(n) 2^{2n - 2\log_2(n)} = 2^{2n-O(\log n)} .
\] 


\section{Rainbow Spanning Trees in $K_{n,m}$} \label{RST_Knm}

We now consider the case where the graph $G$ is complete bipartite. As with the JL-tree associated with a complete graph, the $JL_b$-tree (introduced in Section \ref{jlknm}) associated with a complete bipartite graph captures both the structure of the JL-coloring and the number of RSTs in that coloring.

In this instance, the number of edges with color $\mathcal{C}_1$ associated with an inner vertex $(p,q)$ and children $(p_1, p_2)$ and $(q_1,q_2)$ would be the sum $p_1q_2$ + $p_2q_1$.




Now, we turn our attention to the proof of Theorem \ref{tbip}. We begin by proving the lower bound, followed by the upper bound.
	
	\subsection{The Lower Bound}
	
We first consider the lower bound in Theorem \ref{tbip}. We prove that $|\mathcal{R}(K_{n,m},\varphi)| \geq (n-1)(m-1)+1$ and further, that there exists a coloring achieving this lower bound.

	
	\begin{proof}[Proof of the lower bound in Theorem \ref{tbip}]
		Let $G = K_{n,m}$ be a complete bipartite graph with partitions $N$ and $M$, respectively. By \eqref{convexity!}, for a graph $G$ of order $n$, $|\mathcal{R}(G,\varphi)| \geq |E(G)| - (n-2)$, so it follows that for $K_{n,m}$, $$  |\mathcal{R}(K_{n,m},\varphi)| \geq (n-1)(m-1) + 1.$$ 
We  construct a coloring achieving this bound as follows.  Fix one vertex $a \in N$ and $b \in M$ from each partite set.  Color the edges incident to $a$ and $b$ with distinct colors, and color all other edges the same as the $ab$ edge, so that all color classes except for one have size one.  This coloring has $n+m-1$ colors, it has $n+m-2$ color classes class of size one, and the remaining class has size $nm - (n+m-2) = (n-1)(m-1)+1$.  This coloring is also rainbow cycle free, as any cycle must use two edges of the $ab$ edge's color.  This realizes the bound of \eqref{convexity!} and proves the theorem.


		
	\end{proof}
    
Note that the coloring described above is represented by the $JL_b$-tree where the children of a vertex labeled $(a,b)$ are $(1,b-1)$ and $(a-1, 1)$, respectively. 
	
	\subsection{The Upper Bound}
	
	We now turn our attention to proving the upper bound in Theorem \ref{tbip}. To that end, we let the function $\nu (n,m)$ for $n,m \in \mathbb{N}$ be the maximum number of rainbow spanning trees occurring in any JL-coloring of $K_{n,m}$. We are interested in proving the following theorem:
	
	\begin{theorem}\label{tnu}
		Let $n \geq m \in \mathbb{N}.$ Then 
		
		\[
		\nu (n,m) = m^{n-m+1}((m-1)!)^2.
		\]
	\end{theorem}
	
	Observe that proving Theorem \ref{tnu} proves the upper bound in Theorem \ref{tbip}.
	

	
	\begin{proof}[Proof of the upper bound in Theorem \ref{tbip}]

		The proof proceeds by induction. Observe that the upper bound in Theorem \ref{tbip} holds for the base case $K_{1,1}$. We shall prove it holds for $K_{n,m}$. 
		
		Now, we first claim that for a vertex $(a,b)$ with $a \geq b$ in the $JL_b$-tree, the optimal split for producing the most RSTs is the two vertices $(1,0)$ and $(a-1,b)$. Notice that for $K_{n,m}$ with $n \geq m$, this split yields $n-m+1$ color classes of size $m$ and two color classes of each size $1$ through $m-1$. By the observations made above, this split produces $m^{n+m+1}((m-1)!)^2$ RSTs. 
		
		Now, suppose to the contrary that the split described above does not maximize $|\mathcal{R}(K_{n,m}, \varphi)|$. Then there exists some split, $(n_1, m_1)$ and $(n_2,m_2)$ with $n_1 + n_2 = n$, $m_1 + m_2 = m$, of $(n,m)$ that produces more RSTs. We claim this is not the case. 
		
		To that end, notice that either $n_1 \geq m_1$ or $n_2 \geq m_2$. Without loss of generality, suppose $n_1 \geq m_1$ and observe that by induction, $(n_1,m_1)$ splits in the conjectured optimal way. Thus, the number of RSTs produced by this $(n_1, m_1)$ and $(n_2,m_2)$ split is the following:
		
		\begin{align*}
		(n_1m_2 + n_2m_1)\nu(n_1,m_1)\nu(n_2,m_2) &= (n_1m_2 + n_2m_1)m_1\nu(1,0)\nu(n_1-1,m_1)\nu(n_2,m_2)\\
		 &= m_1(n_1m_2 + n_2m_1)\nu(n_1-1,m_1)\nu(n_2,m_2)
		\end{align*}
		
		Now, the number of RSTs produced by the conjectured optimal split, $(1,0)$ and $(n-1,m)$, is $m \nu(n-1,m)$. Thus, proving our claim is equivalent to showing $m \nu(n-1,m) \geq m_1(n_1m_2 + n_2m_1)\nu(n_1-1,m_1)\nu(n_2,m_2)$. 
		
		To that end, observe that 
		\begin{align*}
		m\nu(n-1,m) &= m\nu((n_1-1)+n_2, m_1+m_2)\\
		&\geq m[(n_1-1)m_2 + n_2m_1]\nu(n_1-1, m_1)\nu(n_2,m_2)
		\end{align*}
		
		where the inequality comes from the fact that $(n_1-1, m_1)$ and $(n_2,m_2)$ might be suboptimal splits for $(n-1,m)$.
		
		Therefore, it is enough to show that
		\begin{equation*}
	m_1(n_1m_2 + n_2m_1)\nu(n_1-1,m_1)\nu(n_2,m_2) \leq m[(n_1-1)m_2 + n_2m_1]\nu(n_1-1, m_1)\nu(n_2,m_2).
		\end{equation*} 
		 Using the fact that $m=m_1+m_2$ and rearranging, this is equivalent to showing that
        
        \[
          0 \leq (n_1 - 1)(m_2)^2 + (n_2 - 1)m_1m_2.
        \]
		
		
		
%
		
		If $n_1, n_2 > 0$, then $(n_1 - 1)(m_2)^2 + (n_2 - 1)m_1m_2 \geq 0$. Now, observe that $n_1 \neq 0$ because we assumed $m_1 \leq n_1$ and $(0,0)$ is not a valid vertex in a $JL_b$-tree. Thus, it remains to consider the case where $n_2 = 0$. If $n_2 = 0$ then $m_2 =1$ and thus,
		\begin{align*}
		(n_1 - 1)(m_2)^2 + (n_2 - 1)m_1m_2
		&= n - 1 - (m-1)\\
		&= n-m\\
		&\geq 0
		\end{align*}
		
		It follows that $(n_1 - 1)(m_2)^2 + (n_2 - 1)m_1m_2 \geq 0$, thus completing the proof.
\end{proof}

\section{General Graphs, General Colorings, and Further Questions}

In this section we briefly investigate a few related questions: How do the results above generalize to arbitrary graphs? How do these results generalize to other $n-1$ colorings, when rainbow cycles are allowed?  We note that there are a myriad of interesting open questions in these areas, some of them raised below, that will likely require new ideas to address.

\subsection{General Graphs with JL-colorings}\label{geng}

As noted in the introduction, the number of rainbow spanning trees in a JL-coloring of a general graph is the product of the sizes of the color classes. In \eqref{convexity!} we observed that by convexity,


\[
|E(G)| - (n-2)= \leq  |\mathcal{R}(G, \varphi)| \leq \left( \frac{|E(G)|}{n-1} \right)^{n-1},
\]

where $\varphi$ is a JL-coloring of $G$.

We have seen that, in the case of a complete bipartite graph, the lower bound is actually achievable.  Furthermore, as also observed in the introduction, a rainbow coloring of any tree meets both bounds.  The following are natural questions which arise when considering the strength of these bounds. 

For the remainder of Section \ref{geng}, assume that all colorings are JL-colorings.

\begin{itemize}
	\item {\bf Sharpness of the lower bound:} For what graphs is there a coloring so that the lower bound is sharp?  Can they be characterized?  
	\item {\bf Sharpness of the upper bound:} Are there any non-trivial examples of the sharpness of the upper bound?  The trivial upper bound given above can be strengthened, somewhat, as the sizes of color classes must be integral.  Let $a_1, \dots, a_{n-1}$ be positive integers so that $\sum a_i = |E(G)|$ and $|a_i - a_{j}| \leq 1$, for $1 \leq i < j \leq n-1$.  Then (applying convexity more carefully), 

	\begin{equation}
	|\mathcal{R}(G,\varphi)| \leq \prod_{i=1}^{n-1} a_i.  \label{impub} 
	\end{equation} 
	For what graphs is there a coloring so that \eqref{impub} is sharp?  
	\item {\bf Graphs maximizing rainbow spanning trees:} Note that for the complete graph, the upper bound \eqref{impub} is \emph{not} satisfied and a coloring maximizing $|\mathcal{R}(K_n,\varphi)|$ does not have all color classes the same size.  This leaves open the possibility that some other $n$-vertex graph $G$ has a coloring so that $\max_\varphi |\mathcal{R}(K_n,\varphi)| \leq \max_\varphi |\mathcal{R}(G,\varphi)|$.  Does such a graph exist?  
\end{itemize} 

We give brief answers, partial in some cases, to these questions.  The first question we can answer precisely and we obtain the following.

\begin{theorem} The lower bound \[
	|E(G)| - (n-2) \leq  |\mathcal{R}(G, \varphi)| 
	\]
	is tight for some coloring iff the graph $G$ can be partitioned into two parts $(X,Y)$ so that $G[X]$ and $G[Y]$ are trees and $|e(X,Y)| \geq 1$.  
\end{theorem} 
\noindent {\bf Remark:} This is not the traditional presentation of $K_{n,m}$, where we have already observed this bound to be tight.  We note, however, that $K_{n,m}$ can also be thought of as two stars, $K_{1,{m-1}}$ and $K_{{n-1},1}$, along with a complete bipartite graph between the leaves and a single edge connecting the roots.  

\begin{proof} 
	If $G$ has the desired form, then one colors each of the trees in a rainbow way, with each color used once and each tree using disjoint sets of colors, and then the bipartite graph on $(X,Y)$ a distinct color.  Then the coloring has no rainbow cycle (as any cycle must use multiple edges of the bipartite graph $(X,Y)$, uses $(|X|-1)+(|Y|-1) + 1 = n-1$ colors, and furthermore, only one color class has size larger than one so the lower bound is realized.  
	
	In the other direction, suppose $G$ has a JL-coloring realizing the lower bound.  Such a coloring has at most one color class of size larger than one.  If each is of size one, $G$ is a tree, which is of the desired form with $X$ and $Y$ being any partition into connected subtrees.  So suppose $G$ is not a tree. Since the coloring is rainbow cycle free, the color classes of size one induce a forest with two components ($X$ and $Y$); and the remaining (larger) color class forms a bipartite graph between them, as desired. 
\end{proof} 

In the complete graph, however, the lower bound is exponential and this leaves many related open questions.  In particular, can one characterize graphs for which this number grows exponentially (or polynomially)?  Is it true, for instance, that in a non-bipartite expander graph $|\mathcal{R}(G,\varphi)|$ is necessarily exponential in the number of vertices?  

We answer the second of questions as follows:

\begin{theorem} 
	Let $G$ be a connected graph and let $a_1, a_2, \dots, a_n $ denote positive integers so that $\sum_{i=1}^{n-1} a_i = |E(G)|$ and $|a_i - a_j| \leq 1$.  Then, as noted above, convexity implies that
	\[
	|\mathcal{R}(G,\varphi)| \leq \prod_{i=1}^{n-1} a_i \leq \left( \frac{|E(G)|}{n-1} \right)^{n-1}.
	\]
If $|E(G)| \geq 2(n-1)$, then the first inequality is strict.  
\end{theorem} 
\noindent{\bf Remark:} This inequality is tight for some coloring when $G$ is a tree, but is also easily seen to be tight when $G$ is unicyclic (that is, when $|E(G)| = n$).  An interesting open question would be to find the largest $|E(G)|$ for an $n$-vertex graph $G$ where this inequality can be tight. 

\begin{proof} 
	If $|E(G)| \geq 2(n-1)$, then the values $a_i$ satisfying the hypothesis of the theorem are all at least two.  On the other hand, the tree decomposition of a JL-coloring described in Section 2, by iteratively partitioning the graph, ends with two parts of size one -- and hence, with a color class of size one.  Thus, in any JL-coloring $|C_i| = 1$ for some $i$ and the bound on the product given is never sharp.
\end{proof} 
This leaves open the rather interesting question of whether there is a general improvement to \eqref{impub}. 

Finally we answer the third question completely with the following.
\begin{theorem}
	If $G$ is an $n$-vertex non-complete graph, then
	\[	\max_\varphi |\mathcal{R}(K_n,\varphi)| > \max_\varphi |\mathcal{R}(G,\varphi)|.
	\]
	
\end{theorem} 
\begin{proof} 
	This follows immediately from the decomposition of JL-colored graphs given in Section 2. Given a graph $G$ and cut $(A,\bar{A})$ in the decomposition of $G$ guaranteed by Proposition \ref{prop1}, increasing the number edges in such a cut gives a JL-colored graph with more edges in the color class (and hence, more rainbow spanning trees).  Iterating eventually gives a JL-colored complete graph. This has more rainbow spanning trees than in $G$, as not all of the cuts augmented were originally complete (as $G$ is not complete).
\end{proof} 

\subsection{General Colorings} 

Another interesting set of questions deals with the case where instead of JL-colorings, one considers general colorings.  As noted in the introduction, if too general colorings are allowed, the question of counting RSTs can become trivial. To this end, for an $n$ vertex graph $G$, let

\begin{align*}
\mathcal{J}(G) &= \{\phi: E(G) \to [n-1]:  \mbox{$\varphi$ is a JL-coloring}\}, \mbox{ and}\\
\mathcal{C}(G) &= \{\phi: E(G) \to [n-1]\}
\end{align*} 
denote the set of JL-colorings and set of general colorings, possibly with rainbow cycles, but restricted to only having $n-1$ colors.  It is easy to see that 
\[
0 = \min_{\varphi \in \mathcal{C}(G)} |\mathcal{R}(G,\varphi)| < \min_{\varphi \in \mathcal{J}(G)} |\mathcal{R}(G,\varphi)|, 
\]
and that this triviality of minimizing the number of rainbow spanning trees continues to hold for graphs with sufficiently many edges, even if the colorings are assumed to be surjective.  

The question of \emph{maximizing} the number of rainbow spanning trees, however, seems quite interesting.  In particular we raise the following question.  

\noindent{\bf Question:} Is it true that 
\[
\max_{\varphi \in \mathcal{C}(K_n)} |\mathcal{R}(K_n,\varphi)| = \max_{\varphi \in \mathcal{J}(K_n)} |\mathcal{R}(K_n,\varphi)| = (n-1)!
\]

The inequality $\max_{\varphi \in \mathcal{J}} |\mathcal{R}(K_n,\varphi)| \leq \max_{\varphi \in \mathcal{C}} |\mathcal{R}(K_n,\varphi)|$ is trivial, as the maximization is over a smaller set. The inequality in the other direction, that $ \max_{\varphi \in \mathcal{J}} |\mathcal{R}(K_n,\varphi)| \geq \max_{\varphi \in \mathcal{C}} |\mathcal{R}(K_n,\varphi)|$ initially appeared unlikely to us, but after some experimentation and thought, it seems plausible. We can show, at least, that colorings with more rainbow spanning trees than the maximizing JL-coloring are quite rare.

\begin{theorem}
	Let $\mathcal{C}'(K_n)  \subseteq \mathcal{C}(K_n)$ denote the set of colorings $\varphi$ of $E(K_n)$ satisfying $\mathcal{R}(K_n,\varphi) \geq (n-1)!$.  Then
	\[
	\lim_{n \to \infty} \frac{|\mathcal{C}'(K_n)|}{|\mathcal{C}(K_n)|} = 0.
	\]  \label{MTT}
\end{theorem}  
\begin{proof} 
	Let $\varphi$ denote a uniform random coloring of the edges of $K_n$ so that the color of each edge is independently and uniformly chosen from $[n-1]$.  For a fixed spanning tree $T$, the probability that $T$ is rainbow is $(n-1)!/(n-1)^{n-1}$.  Then by Cayley's formula and linearity of expectation
	\[
	\mathbb{E}\Big[|\mathcal{R}(K_n,\varphi)|\Big] = n^{n-2} \frac{(n-1)!}{(n-1)^{n-1}} = \left(\frac{n}{n-1}\right)^{n-2} \cdot \frac{1}{n-1} \cdot (n-1)! \leq \frac{e}{n-1} \cdot  (n-1)!.  
	\]	
	The result then follows by Markov's inequality, as 
	\[
	\frac{|\mathcal{C}'(K_n)|}{|\mathcal{C}(K_n)|} = \p\Big(|\mathcal{R}(K_n,\varphi)| \geq (n-1)! \Big) \leq  \frac{e}{n-1} \to 0.
	\]
\end{proof} 

In general, understanding $|\mathcal{R}(G,\varphi)|$ for an arbitrary $\varphi \in \mathcal{C}$ seems difficult.  It is clear that, if $\mathcal{C}_1, \dots, C_{n-1}$ are the color classes of $\varphi$, then 
\[
|\mathcal{R}(G,\varphi)| \leq \prod_{i=1}^{n-1} |\mathcal{C}_i|.
\]
The inequality is strict when collections of $n-1$ edges, one of each color, include cycles.  Understanding these collections in a simple way, however, seems difficult.  

As a first step in this direction, we observe that we can prove an analogue of the matrix tree theorem of Kirchoff, which gives a way of counting rainbow spanning trees in a general graph.  

Recall that the combinatorial Laplacian matrix of a graph is
the matrix 
\[
L = D - A, 
\]
where $D$ is a diagonal matrix consisting of vertex degrees and $A$ is the adjacency matrix.  Then the matrix tree theorem states that the determinant of any cofactor of $L$ is the number of spanning trees in this graph.  

We generalize this result to colored graphs.  Because we deal with $n-1$ edge colored graphs, and because the statement is cleaner in this case, we focus on the $n-1$ colored case.  Given a graph $G$ and an edge coloring $\varphi: E(G) \to [n-1]$, we define the colored graph Laplacian $L_\phi$ of $G$ so that
\[
[L_{\varphi}]_{ij} = \left\{ \begin{array}{cl} 0 & \mbox{ if $i \neq j$,  $v_i \not\sim v_j$} \\
-c_{\varphi(v_iv_j)} & \mbox{ if $i \neq j$ and $v_i \sim v_j$}\\
\sum\limits_{k: v_i \sim v_k} c_{\varphi(v_iv_j)} & \mbox{ if $i=j$}
\end{array} \right.,
\]
where $c_{i}$ for $i = 1, \dots, n-1$ are indeterminates.   Note that if one sets $c_i = 1$, for all $i$, then one recovers the ordinary graph Laplacian, as above.   

\begin{theorem}[Matrix Tree Theorem for Rainbow Spanning Trees] 
	Let $G$ be a graph and $\varphi:E(G) \to [n-1]$ an edge coloring of G.  Let $L_\varphi$ of $G$ be the colored graph Laplacian defined above.  Let $L'$ denote a principle cofactor of $L_\varphi(G)$ and
	\[
	f(c_1, \dots, c_{n-1}) = \det L'.  
	\]
	Then \[
	|\mathcal{R}(G, \varphi)| = [f(c_1, \dots, c_{n-1})]_{c_1c_2\dots c_{n-1}} = \frac{\partial}{\partial c_1} \frac{\partial}{\partial c_2} \cdots \frac{\partial}{\partial c_{n-1} }  \det L'.  
	\]
\end{theorem}
\noindent{\bf Remark:} The proof, a simple modification of the usual proof of the matrix tree theorem, actually shows that $\det L'$ is a generating function for different colorings of spanning trees.  This remains true for colorings with more than $n-1$ colors.  Rainbow spanning trees, in this setting, are counted by the coefficients of squarefree terms.  The advantage in stating the $n-1$ color case is that there is only one such term.   

\begin{proof} 
	The proof largely follows that of the standard matrix tree theorem.  
	
	Let $B_\varphi$ be a $|V| \times |E|$ matrix, indexed by vertices and edges respectively.  The column indexed by edge $v_iv_j$  has non-zero entries only in the $v_i$ and $v_j$ positions: one of these is set to be $\sqrt{c_{\varphi(v_iv_j)}}$ and the other $-\sqrt{c_{\varphi(v_iv_j)}}$, with the signing chosen arbitrarily.  Then it is easy to check that
	\[
	L_\varphi = B_\varphi B_\varphi^T,
	\]
	just as with the standard Laplacian.  If the $v_i$th row and column of the Laplacian are removed, then $L' = B'(B')^T$, where $B'$ is obtained by removing the $v_i$th row of $B_{\varphi}$.  
	
	Then, by the Cauchy-Binet formula, 
	\begin{align*}
	f(c_1, \dots, c_{n-1}) = \det L' &= \det (B')(B')^T
	\\
	&= \sum_{\substack{A \subset E(G) \\ |A|=n-1}} \det(B'|_{A}) \det((B')^T|_A)  \\
	&= \sum_{\substack{A \subset E(G) \\ |A|=n-1}} \det(B'|_{A})^2,
	\end{align*}
	and it is straightforward to verify that
	\[
	\det(B'|_A) = \left\{ \begin{array}{cl} 0 &\mbox{ if the edges in $A$ contain a cycle} \\
	\pm \prod_{e \in A} \sqrt{c_{\varphi(e)}} & \mbox{ if the edges in $A$ form a spanning tree} \end{array}\right. 
	\]
	Thus,
	\[
	f(c_1, \dots, c_{n-1}) = \sum_{\substack{T~spanning\\tree~ of~G}} \prod_{e \in T} c_{\varphi(e)}.    
	\]
	Then the number of rainbow spanning trees is exactly the coefficient of the monomial where each of the $c_i$s has degree one, as claimed.  As this polynomial is homogenous of degree $n-1$ in the variables $c_i$, the coefficient can be recovered by iteratively taking derivatives.
\end{proof}

\end{document}